\documentclass[12pt,reqno]{article}
\usepackage{amsmath} 

\usepackage{color}

\newcommand{\N}{\mathbb{N}}

\newcommand{\Z}{\mathbb{Z}}

\newcommand{\F}{\mathbb{F}}

\usepackage[english]{babel}
\usepackage[latin1]{inputenc}
\usepackage{bbm}
\usepackage{cite}

\usepackage{amssymb}
\usepackage{amscd}

\usepackage{fullpage}
\usepackage{float}

\usepackage{lineno}

\usepackage{amsthm}
\usepackage{amsfonts}
\usepackage{latexsym}
\usepackage{epsfig}
\usepackage{mathrsfs}
\usepackage{tikz}
\usepackage{pifont}
\usepackage{hyperref}

\theoremstyle{plain}
\newtheorem{theorem}{Theorem}[section]

\newtheorem{conjecture}[theorem]{Conjecture}
\newtheorem{corollary}[theorem]{Corollary}

\theoremstyle{definition}

\newtheorem{example}[theorem]{Example}

\theoremstyle{remark}

\theoremstyle{question}
\newtheorem{question}[theorem]{Question}
\numberwithin{equation}{section}

\DeclareMathAlphabet{\mathbbold}{U}{bbold}{m}{n}

\allowdisplaybreaks

\author{
Charles Burnette \\
Department of Mathematics \\
Xavier University of Louisiana \\
New Orleans, LA 70125 \\
\href{mailto:cburnet2@xula.edu}{\texttt{cburnet2@xula.edu}}
}
\title{On power maps over weakly periodic rings}

\begin{document}

\thispagestyle{empty}

\maketitle

\begin{abstract}
A ring $R$ is called \textit{weakly periodic} if every $x \in R$ can be written in the form $x = a + b,$ where $a$ is nilpotent and $b^m = b$ for some integer $m > 1.$ The aim of this note is to consider when a nonzero nilpotent element $r$ is the period of some power map $f(x) = x^n,$ in the sense that $f(x + r) = f(x)$ for all $x \in R,$ and how this relates to the structure of weakly periodic rings.

In particular, we provide a new proof of the fact that weakly periodic rings with central and torsion nilpotent elements are periodic commutative torsion rings. We also prove that $x^n$ is periodic over such rings whenever $n$ is not coprime with each of the additive orders of the nilpotent elements. These are in fact the only periodic power maps over finite commutative rings with unity. Finally, we describe and enumerate the distinct power maps over Corbas $(p, k, \phi)$-rings, Galois rings, $\Z/n\Z,$ and matrix rings over finite fields.\\


\end{abstract}

\section{Introduction}

Throughout, $R$ is a ring (not necessarily commutative or unital) and $R^+$ is the additive group of $R.$ By order of a ring element we will always mean its group-theoretic order as a member of $R^+.$ The set of nilpotent elements of $R$ will be denoted by $\textup{Nil}(R).$ The index of a nilpotent $x$ is the smallest positive integer $n$ such that $x^n = 0,$ and the index of $\textup{Nil}(R)$ is the largest index, if extant, among all elements of $\textup{Nil}(R).$ Any additional terminology and notation not explicitly defined herein are standard in the literature.

A ring $R$ is called \textit{periodic} if for each $x \in R$ the set $\{x^n : n \in \N\}$ is finite. Equivalently, for each $x \in R,$ there are positive integers $m(x)$ and $n(x)$ such that $x^{m(x) + n(x)} = x^{m(x)}.$ If the aforementioned $n(x)$ can be taken to be constant across all $x,$ then the smallest such constant permissible is called the \textit{exponential period} of $R,$ which we will denote by $\mu_1 := \mu_1(R),$ and the \textit{onset of exponential periodicity} of $R,$ which we will denote by $\mu_0 := \mu_0(R)$, is given by
\[\mu_0 = \max_{x \in R}\min\{m \in \N : x^{m + \mu_1} = x^m\},\]
provided this maximum exists. Otherwise, we say that $\mu_1$ and $\mu_0$ are infinite. Notice that $\mu_0$ is no smaller than the index of $\textup{Nil}(R)$ and that, for rings with unity, $\mu_1$ is no smaller than the exponent of the group of units $R^{\times}.$

An element $x \in R$ is called potent if $x^m = x$ for some integer $m > 1.$ Let $\textup{Pot}(R)$ denote the set of potent elements of $R.$ A ring $R$ is called \textit{weakly periodic} if $R = \textup{Nil}(R) + \textup{Pot}(R).$ We remark that if $R = \textup{Nil}(R) \cup \textup{Pot}(R),$ then $\mu_0$ is clearly the index of $\text{Nil}(R).$ If $R = \textup{Pot}(R),$ then $R$ is called a $J$-\textit{ring}. It is well known that every $J$-ring is commutative \cite{Jacobson}.

Some obvious examples of periodic rings are finite rings, Boolean rings, and nil rings. In fact, finite fields and Boolean rings are $J$-rings. Bell \cite{Bell} proved that periodic rings are weakly periodic, but the status of the converse is still unresolved. Periodicity has been established for several special classes of weakly periodic rings though (e.g. \cite{AHY}, \cite{BK}, \cite{BT}, \cite{GTY}, \cite{Yaqub}).

The power maps over a ring are the functions of the form $f(x) = x^n$ for some fixed $n \in \N.$ In this note, we assess what power map ``oscillations'' may inform us about the arithmetical structure of weakly periodic rings. In particular, we will show that if the nilpotent elements of $R$ are additively torsion and multiplicatively central, then each one is a period of some power map. This leads us to an alternative proof of the fact that weakly periodic rings with this property are periodic commutative torsion rings. A mild generalization of this was proved by Bell and Tominaga in \cite{BT}, who relied on Pierce decompositions and Chacron's periodicity criterion \cite{Chacron}. We will also prove that $x^n$ is periodic over such rings whenever $n$ is not coprime with each of the orders of the nilpotent elements. This will allow us to derive tight lower and upper bounds for the number of periodic power maps over commutative rings when $\mu_1$ and $\mu_0$ are both finite. Assorted examples are peppered throughout to help illustrate the results. For ease of navigation, Table 1 catalogs the most instructive and concrete examples featured in this paper.

\begin{table}
\centering
\caption{Periodic power map enumerations for some rings}
\resizebox{\columnwidth}{!}{%
\begin{tabular}{l | l  l  l}
\hline
\textit{Ring $R$} & $\mu_1(R)$ & $\mu_0(R)$ & \# \textit{of periodic power maps} \\ \hline \hline
\multicolumn{1}{c|}{ }&\multicolumn{3}{|c}{ }\\[-0.4cm]
\begin{tabular}{@{}l@{}} Corbas \\ $(p, k, \phi)$-ring \\ (pp. 3--4, 9) \end{tabular} & $p(p^k - 1)$ & 2 & \begin{tabular}{@{}l@{}} $p^k - 1$ if $\phi$ is the identity, \\ 0 otherwise \end{tabular} \\[0.4cm] \hline
\multicolumn{1}{c|}{ }&\multicolumn{3}{|c}{ }\\[-0.4cm]
\begin{tabular}{@{}l@{}} $R_1 \times R_2$, \\ $R_1$ reduced, \\ $R_2 \neq 0$ nil \\ (p. 5) \end{tabular} & $\mu_1(R_1)$ & $\mu_0(R_2) = \text{index of}\ R_2$ & \begin{tabular}{@{}l@{}}at least $\mu_1(R_1),$ \\ at most $\mu_1(R_1) + \mu_0(R_2) - 2$ \end{tabular} \\[0.4cm] \hline
\multicolumn{1}{c|}{ }&\multicolumn{3}{|c}{ }\\[-0.4cm]
\begin{tabular}{@{}l@{}} $\text{GR}(p^k, d),$ $d \geq 2$ \\ (p. 10) \end{tabular} & $p^{k - 1}(p^d - 1)$ & $k$ & \begin{tabular}{@{}l@{}} $p^{k - 2}(p^d - 1) + \left\lfloor\frac{k - 1}{p}\right\rfloor$ if $k \geq 2,$ \\ 0 if $k = 1$ \end{tabular} \\[0.4cm] \hline
\multicolumn{1}{c|}{ }&\multicolumn{3}{|c}{ }\\[-0.4cm]
\begin{tabular}{@{}l@{}} $\Z/n\Z$ \\ (pp. 10--11) \end{tabular} & \begin{tabular}{@{}l@{}} Carmichael \\ function $\lambda(n)$ \end{tabular} & \begin{tabular}{@{}l@{}} $E(n) :=$ maximal exponent \\ in prime factorization of $n$ \end{tabular} & $\sum\limits_{\substack{d\, |\, \textup{rad}(n/\text{rad}(n)) \\ d \neq 1}} (-1)^{\omega(d) + 1}\!\left\lfloor\frac{\lambda(n) + E(n) - 1}{d}\right\rfloor$ \\[0.4cm] \hline
\multicolumn{1}{c|}{ }&\multicolumn{3}{|c}{ }\\[-0.4cm]
\begin{tabular}{@{}l@{}} $\mathcal{M}_n(\F_q)$ \\ (pp. 11--12) \end{tabular} & \begin{tabular}{@{}l@{}} the exponent \\ of $\text{GL}(n, \F_q)$ \end{tabular} & $n$ & 0 
\end{tabular}
}
\end{table}

\section{Periodic Power Maps}

If $G$ is an additive group and $X$ is a set, then a function $\alpha : G \rightarrow X$ is \textit{periodic} if there exists an $h \in G\backslash \{0\}$ such that $\alpha(g + h) = \alpha(g)$ for all $g \in G.$ Such an element $h$ is called a \textit{period} of $\alpha.$ It is easy to see that the periods of a function together with $0$ form a subgroup of $G.$ We will call this subgroup $\text{Per}(f).$ Clearly, a power map over a ring $R$ is periodic as a function on $R^+$ only if $\text{Per}(f) \subseteq \text{Nil}(R).$ On the other hand, nilpotence alone does not guarantee that an element is the period of some power map. Here are a couple counterexamples.

\begin{example}
Let $R = \prod_{i \in \N} \Z/p_i^2\Z,$ where $p_i$ is the $i^{\text{th}}$ prime number. The sequence $(p_i)_{i \in \N}$ is nilpotent of index 2. Therefore, if $(x_i)_{i \in \N} \in R,$ then
\begin{equation}
\label{nottorsion}
\left((x_i + p_i)^n\right)_{i \in \N} = \left(x_i^n + np_ix_i^{n - 1}\right)_{i \in \N}
\end{equation}
for all $n \in \N.$ Since $(p_i)_{i \in \N}$ has infinite order, there is no integer $n$ such that $x_i^n + np_ix_i^{n - 1} = x_i^n$ for every $i \in \N.$
\end{example}

\begin{example}
If $p$ is a prime number, $k$ is a positive integer, and $\phi$ is an automorphism of the Galois field $\F_{p^k},$ then the Corbas $(p, k, \phi)$-ring \cite{Corbas} is the ring $R$ in which $R^+ = \F_{p^k} \oplus \F_{p^k}$ and the ring multiplication $\cdot$ is defined by
\[(a, b) \cdot (c, d) = (ac, ad + b\phi(c)).\]
It is straightforward to confirm that $R$ satisfies the following properties:
\begin{itemize}
\item For all $n \in \N,$
\begin{equation}
(a, b)^n = \left(a^n, b\phi(a^{n - 1})\sum_{j = 0}^{n-1} a^j\phi(a^{-j})\right).
\end{equation}
If $\phi$ is not the identity automorphism and $a \neq 0,$ then this can be simplified to
\begin{equation}
(a, b)^n = \left(a^n, b\phi(a^{n - 1})\frac{a^n\phi(a^{-n}) - 1}{a\phi(a^{-1}) - 1}\right),
\end{equation}

\item $\textup{Nil}(R) = 0 \times \F_{p^k}$ has index 2,

\item $R$ is commutative if and only if $\phi$ is the identity.
\end{itemize}
Hence, for each $a, b, y \in \F_{p^k}$ and $n \in \N,$
\begin{equation}
\label{noncommutative}
\left((a, b) + (0, y)\right)^n = \left(a^n, (b + y)\sum_{j = 0}^{n-1} a^j\phi(a^{n - j - 1})\right).
\end{equation}
If $\phi$ is not the identity, then $0$ is the only value of $y$ for which $\left((a, b) + (0, y)\right)^n$ is identically $(a, b)^n.$ In this case, $R$ has no periodic power maps whatsoever.
\end{example}

Unboundedness of order and noncommutativity are responsible for the pathologies of Examples 2.1 and 2.2, respectively. However, as long as $\textup{Nil}(R)$ circumvents these traits, we can promise that each nonzero nilpotent element is a period of some power map. We will refer to such rings as \textit{nilperiod}.

\begin{theorem}
\label{nilpotentsperiodic}
If $R$ is a ring in which every nilpotent element is central and torsion, then $R$ is nilperiod.
\end{theorem}

\begin{proof}
Let $r \in \textup{Nil}(R),$ and let $i$ and $j$ be the index and order, respectively, of $r.$ If we set $n = (i-1)!j,$ then for all $x \in R,$
\begin{equation}
\label{binomial}
(x + r)^n = x^n + \sum_{k = 1}^{i - 1} \binom{n}{k}r^kx^{n - k}.
\end{equation}
Since the binomial coefficient $\binom{n}{k}$ is a multiple of $n/\textup{gcd}(n, k)$ for each integer $k$ (cf Problem B2 of the 2000 Putnam competition \cite{KPV}) and $n/\textup{gcd}(n,k)$ is a multiple of $j$ for each $k \in [i - 1],$ the sum in (\ref{binomial}) vanishes. Therefore $(x + r)^n = x^n.$
\end{proof}

\begin{corollary}
\label{weaklyperiodicisperiodic}
If $R$ is a weakly periodic ring in which every nilpotent element is central and torsion, then $R$ is a periodic commutative torsion ring.
\end{corollary}

\begin{proof}
Let $x \in \textup{Pot}(R)$ and $y \in \textup{Nil}(R)$ be given. By Theorem \ref{nilpotentsperiodic}, there is an integer $n$ such that $(x + y)^n = x^n.$ Since there is also an integer $m > 1$ such that $x^m = x,$ we can see that
\begin{equation}
(x + y)^{nm} = x^{nm} = x^n = (x + y)^n.
\end{equation}
Hence $R$ is periodic, and because all nilpotent elements are central, we can further conclude that $R$ is commutative (see \cite{Herstein1}). Consequently, $\textup{Nil}(R)$ is an ideal.

To see that $R$ is torsion, first notice that $R/\textup{Nil}(R)$ is a $J$-ring. As Jacobson explained in the proof of his classic ``$a^n = a$'' theorem \cite{Jacobson}, the additive group of a $J$-ring is torsion. This implies that for each $x \in \textup{Pot}(R),$ there is a positive integer $j$ such that $jx \in \textup{Nil}(R).$ Since every nilpotent element of $R$ is  torsion, it follows that $R$ itself is torsion as a whole.
\end{proof}

It may be interesting to figure out the extent to which the conditions of Theorem \ref{nilpotentsperiodic} can be loosened. Here is an example demonstrating that either of the hypotheses in Theorem \ref{nilpotentsperiodic} can be resoundingly defied.

\begin{example}
\label{hypothesisrejection}
Consider the direct product $R = R_1 \times R_2,$ where $R_1$ is a reduced ring and $R_2$ is a nil ring of finite index, say $n.$ Then $\textup{Nil}(R) = 0 \times R_2$ has index $n$ as well. Furthermore,
\begin{equation}
\left((a, b) + (0, y)\right)^m = (a^m, 0) = (a, b)^m
\end{equation}
for all $a \in R_1,$ $b, y \in R_2,$ and integers $m \geq n.$
\end{example}

Take $R_2$ to be a torsion-free abelian group equipped with the zero multiplication to see that nilperiod rings can be entirely devoid of torsion nilpotent elements. Further still, the existence of noncommutative nil rings rules out the necessity of central nilpotent elements. However, nilperiod rings seem constrained enough to compel $\textup{Nil}(R)$ to be an ideal. We call $R$ an \textit{NI-ring} if $\textup{Nil}(R)$ is an ideal. Equivalently, $\text{Nil}(R)$ coincides with the upper nilradical $\text{Nil}^*(R)$ (i.e. the sum of all nil ideals of $R$).

\begin{conjecture}
\label{nilperiodideal}
Every nilperiod ring is an \textup{NI}-ring.
\end{conjecture}

If Conjecture \ref{nilperiodideal} is correct, then noncommutative weakly periodic nilperiod rings are ``almost'' commutative in the sense that their commutator ideals are nil. Indeed, for if $R$ is is a weakly periodic \textup{NI}-ring, then $R/\textup{Nil}(R)$ is commutative due to being a $J$-ring. Moreover, because the Jacobson radical $\text{J}(R)$ of a weakly periodic ring is nil, we could report that $\textup{J}(R) = \text{Nil}(R) = \text{Nil}^*(R)$ and that $R$ is ultimately periodic \cite[Lemma 1]{GTY}.

Of course in noncommutative rings, $\text{Nil}^*(R)$ typically differs from the lower nilradical $\text{Nil}_*(R),$ that is, the intersection of all the prime ideals of $R.$ A ring $R$ is called \textit{2-primal} if $\text{Nil}_*(R) = \text{Nil}(R).$ Marks \cite{Marks} provided a thorough list of conditions on noncommutative rings that enforce 2-primality together with their interdependencies, but none involve weak periodicity. Furthermore, being 2-primal is a necessary and sufficient condition for a ring with bounded nilpotency index to be an NI-ring \cite[Proposition 1.4]{HJL}. We should therefore suspect that a weakly periodic NI-ring may fail to be 2-primal provided its nilpotency index is unbounded.

\begin{example}
Let $S$ be a finite 2-primal ring, $n \in \N,$ and $R_n$ be the $2^n \times 2^n$ upper triangular matrix ring over $S.$ Each $R_n$ is finite and thus periodic. Proposition 2.5 of \cite{BHL} further implies that each $R_n$ is 2-primal, and so they are NI-rings. Now embed each $R_n$ into $R_{n + 1}$ via the monomorphism $\sigma$ defined by
\[\sigma(A) = \begin{pmatrix} A & 0 \\ 0 & A \end{pmatrix}.\]
Then $\mathcal{D} = \langle R_n, \sigma_{nm}\rangle,$ with $\sigma_{nm} = \sigma^{m - n}$ whenever $n \leq m,$ is a direct system over $\N.$ Set $R = \varinjlim R_n,$ the direct limit of $\mathcal{D}.$ Since $R = \bigcup_{n = 1}^{\infty} R_n$ is the union of periodic rings, $R$ is itself periodic. However, \cite[Proposition 1.1 and Example 1.2]{HJL} explains why $R$ fails to be 2-primal, even though the direct limit of a direct system of NI-rings is itself an NI-ring.
\end{example}

We can at least verify Conjecture \ref{nilperiodideal} for polynomial identity algebras over fields of characteristic zero. The proof adapts techniques employed by Herstein in \cite{Herstein2} and \cite{Herstein3}.

\begin{theorem}
\label{PIalgebra}
If $R$ is a nilperiod \textup{PI}-algebra over a field $K$ of characteristic zero, then $R$ is an \textup{NI}-ring.
\end{theorem}

\begin{proof}
Let $a, b \in \textup{Nil}(R).$ Then there are positive integers $\ell$ and $m$ such that $(x + a)^m = x^m$ and $(x + b)^{\ell} = x^{\ell}$ for all $x \in R.$ Consequently, $(a + b)^{\ell m} = 0,$ and so $a + b$ is nilpotent.

Next, suppose $c \in \textup{Nil}(R),$ and let $r \in R$ be arbitrary. Let $S$ be the subalgebra of $R$ generated by $c$ and $r.$ Since $S$ is a finitely-generated PI-algebra, $\textup{J}(S)$ is nil due to Amitsur's Nullstellensatz \cite{Amitsur}, \cite{Braun}.

Now assume, to the contrary, that $c \not\in \textup{J}(S).$ Then the coset $\overline{c} = (c + \textup{J}(S)) \in S/\textup{J}(S)$ is a nonzero nilpotent element of index, say, $j.$ Because $S/\textup{J}(S),$ as a semiprimitve ring, is a subdirect product of primitive rings $S_i,$ each of which is a homomorphic image of $S/\textup{J}(S),$ the coset $\overline{c}$ projects to a nilpotent element $\nu_i$ within each factor $S_i.$ Note that not all of these $\nu_i$ can be 0, otherwise $\overline{c} = 0,$ which would indicate that $c \in \textup{J}(S).$

Let $S_i$ be a factor in which $\nu_i \neq 0,$ and suppose that $\nu_i$ is nilpotent of index $j_i.$ Then $2 \leq j_i \leq j,$ the power $\nu_i^{j_i - 1}$ is nilpotent of index 2, and $\overline{c^{j_i - 1}} \neq 0,$ which implies that $c^{j_i - 1} \not\in \textup{J}(S).$ But since $c^{j_i - 1}$ is then a nonzero nilpotent element of $R,$ by hypothesis $c^{j_i - 1}$ is the period of some power map $f(x) = x^{n_i}.$ As a result, the factor $S_i$ inherits the identity $(x + \nu_i^{j_i - 1})^{n_i} = x^{n_i}.$ Furthermore, $S_i$ cannot be a division ring as it contains the nonzero nilpotent $\nu_i.$ Due to Jacobson's density theorem, $S_i$ is thus isomorphic to a dense subring of $\textup{End}(V_i)$ for some vector space $V_i$ over a division ring $D_i.$ Since $S_i$ is itself a PI-algebra, we may assume that it is finite-dimensional (cf \cite[Lemma 5 and Theorem 1]{Kaplansky}), and so $S_i \cong \mathcal{M}_{s_i}(D_i),$ the ring of $s_i \times s_i$ matrices with entries in $D_i,$ for some integer $s_i \geq 2.$ So if $I_{s_i} \in S_i$ corresponds to the $s_i \times s_i$ identity matrix, then
\begin{equation}
\left(I_{s_i} + \nu_i^{j_i - 1}\right)^{n_i} = I_{s_i} = I_{s_i} + n_i\nu_i^{j_i - 1},
\end{equation}
which implies that $\nu_i^{j_i - 1}$ is torsion.

Because every $j_i$ must lie between 2 and $j,$ we need only finitely many different values for the $n_i.$ Hence, $\overline{c}$ is torsion. Therefore $kc \in \textup{J}(S)$ for some integer $k \geq 2.$ Yet since $K$ is a field of characteristic zero and $\textup{J}(S)$ is closed under scalar multiplication, we find that
\begin{equation}
c = (k1_{K})^{-1}kc \in \textup{J}(S),
\end{equation}
which is at odds with our original assumption. It thus follows that $c \in \textup{J}(S)$ and so $cr$ and $rc$ are in $\textup{J}(S)$ as well. Since $\textup{J}(S)$ is nil, both $cr$ and $rc$ are nilpotent, as required.
\end{proof}

\begin{corollary}
If $R$ is a periodic nilperiod algebra over a field of characteristic zero with finite $\mu_1$ and $\mu_0,$ then $R$ is an \textup{NI}-algebra.
\end{corollary}

\begin{proof}
Since $x^{\mu_0 + \mu_1} = x^{\mu_0}$ for all $x \in R,$ it follows that $R$ is a PI-algebra, and so Theorem \ref{PIalgebra} applies.
\end{proof}

\section{The Number of Distinct Periodic Power Maps}

We now restrict our attention to periodic rings with finite $\mu_1$ and $\mu_0,$ for which there are $\mu_0 + \mu_1 - 1$ distinct power maps over $R$\footnote{The only exception to this is the zero ring. This is because $\mu_0, \mu_1 \geq 1,$ but the zero ring only has one mapping on it.}.  Let $\mu_{P}: =\mu_{P}(R)$ denote the number of distinct periodic power maps over $R.$ Pinpointing the precise ring properties that determine the value of $\mu_P$ is currently beyond our grasp. However, a lower bound for $\mu_P$ is quite tenable, provided that $R$ is commutative and $\text{Nil}(R)^+$ is a torsion group with finite exponent. In essence, nilpotent elements of simultaneously prime order and index 2 are the ``fundamental'' periods of the easiest-to-distinguish periodic power maps. First, we borrow some notation from number theory. The squarefree radical of a positive integer $n,$ denoted $\textup{rad}(n),$ is the product of the distinct prime factors of $n,$ and $\omega(n)$ is the number of distinct prime factors of $n.$ Despite the threat of notational confusion, we will let $\mu(n)$ be the M\"{o}bius function (because one can never get enough of the Greek letter mu).

\begin{theorem}
\label{periodicpowermapbound}
Let $R$ be a commutative ring in which the exponent of $\textup{Nil}(R)^+$ is finite. If $N$ is the least common multiple of the orders of the nilpotent elements, then $ f(x) = x^n$ is periodic for every integer $n \in \N$ that is not coprime with $N.$ Accordingly, 
\[\mu_P \geq -\sum_{\substack{d\, |\, N \\ d \neq 1}} \mu(d)\left\lfloor\frac{\mu_0 + \mu_1 - 1}{d}\right\rfloor = \sum_{\substack{d\, |\, \textup{rad}(N) \\ d \neq 1}} (-1)^{\omega(d) + 1} \left\lfloor\frac{\mu_0 + \mu_1 - 1}{d}\right\rfloor\]
if $\mu_1$ and $\mu_0$ are finite. This lower bound is attained if $R$ is also finite and unital.
\end{theorem}

\begin{proof}
By Cauchy's theorem for abelian groups, if $p$ is a prime number that divides $N,$ then $\textup{Nil}(R)$ contains a necessarily nonzero element $r$ of order $p.$ If the index of $r$ is $k,$ then the index of $r^{k - 1}$ is 2, and since $r\, |\, r^{k - 1},$ the order of $r^{k - 1}$ is also $p.$ We can thus invoke the freshman's dream to see that
\begin{equation}
(x + r^{k - 1})^{pm} = (x^p + r^{(k - 1)p})^m = x^{pm} 
\end{equation}
for every $x \in R$ and $m \in \N.$

So if $pm \leq \mu_0 + \mu_1 - 1,$ then $f(x) = x^{pm}$ is a nonrepetitive periodic power map over $R.$ We can enumerate all these maps in two different ways. One is a routine application of the principle of inclusion-exclusion. The other way is to note that the number of integers no larger than $\mu_0 + \mu_1 - 1$ which are coprime to $N$ is given by
\begin{align}
\sum_{\substack{1 \leq n \leq \mu_0 + \mu_1 - 1 \\ \gcd(n, N) = 1}} 1 &= \sum_{n = 1}^{\mu_0 + \mu_1 - 1} \sum_{d\, | \gcd(n, N)} \mu(d)
\\ &= \sum_{d\, |\, N} \mu(d) \sum_{\substack{1 \leq n \leq \mu_0 + \mu_1 - 1 \\ d\, |\, n}} 1
\\ &= \sum_{d\, |\, N} \mu(d)\left\lfloor\frac{\mu_0 + \mu_1 - 1}{d}\right\rfloor = \mu_0 + \mu_1 - 1 + \sum_{\substack{d\, |\, N \\ d \neq 1}} \mu(d)\left\lfloor\frac{\mu_0 + \mu_1 - 1}{d}\right\rfloor
\\ &= \mu_0 + \mu_1 - 1 + \sum_{\substack{d\, |\, \text{rad}(N) \\ d \neq 1}} (-1)^{\omega(d)}\left\lfloor\frac{\mu_0 + \mu_1 - 1}{d}\right\rfloor.
\end{align}

Lastly, let us see that these are the only periodic power maps if $R$ is a finite commutative ring with unity. Since finite rings are trivially Artinian, $R \cong \prod_i R_i,$ where the $R_i$ are finite commutative local rings, and for each $i$ we have a natural surjective homomorphism $\pi_i: R \rightarrow R_i.$ Now suppose that $g(x) = x^n$ is periodic and $s \in \text{Per}(g)\backslash\{0\}.$ Then $\pi_is \neq 0_{R_i}$ for some $i.$ Furthermore,
\begin{equation}
(1_{R_i} + \pi_is)^n = 1_{R_i} = 1_{R_i} + \sum_{k = 1}^{n} \binom{n}{k}\pi_is^k,
\end{equation}
and so
\begin{equation}
\sum_{k = 1}^{n} \binom{n}{k}\pi_is^k = \pi_is\left(n1_{R_i} + \sum_{k = 2}^{n} \binom{n}{k}\pi_is^{k - 1}\right) = 0_{R_i}.
\end{equation}
Therefore $n1_{R_i} + \sum_{k = 2}^{n} \binom{n}{k}\pi_is^{k - 1} = 0_{R_i}$ or is a zero divisor. Either way, $n1_{R_i}$ must be a non-unit due to the locality of $R_i.$ It follows that $n$ is not coprime to $\text{char}(R_i),$ otherwise B\'{e}zout's identity could be used to express the unit $1_{R_i}$ as a $\Z$-linear combination of the non-units $n1_{R_i}$ and $\text{char}(R_i)1_{R_i}.$ However, $\text{char}(R_i)$ must be a prime power $p^t,$ and so $p$ is a common divisor of $n$ and $\text{char}(R_i).$ Since the order of $s$ is a multiple of $\text{char}(R_i),$ we conclude that $n$ and the order of $s$ are not coprime.
\end{proof}

Observe that $\mu_0 = 1$ for $J$-rings and $\mu_1 = 1$ for nil rings. This quickly leads to the two following corollaries.

\begin{corollary}
\label{jring}
Let $R$ be a $J$-ring with finite $\mu_1.$ Then $R$ has $\mu_1$ distinct power maps, none of which are periodic.
\end{corollary}

\begin{corollary}
If $R$ is a nil ring of finite index $\mu_0,$ then $R$ has $\mu_0$ distinct power maps. If, in addition, $R$ is also commutative and of bounded torsion, then
\[\mu_P \geq -\sum_{\substack{d\, |\, N \\ d \neq 1}} \mu(d)\left\lfloor\frac{\mu_0}{d}\right\rfloor = \sum_{\substack{d\, |\, \textup{rad}(N) \\ d \neq 1}} (-1)^{\omega(d) + 1} \left\lfloor \frac{\mu_0}{d}\right\rfloor,\]
where the notation of Theorem \ref{periodicpowermapbound} has been reprised.
\end{corollary}

\section{Miscellaneous Examples}

\subsection{Weakly periodic rings that annihilate $\textup{Nil}(R)$}

Theorem 2.3 remains true if we swap out the torsionality of $\textup{Nil}(R)$ for the condition that $xr = 0 = rx$ for all $x \in R$ and $r \in \textup{Nil}(R).$ (The nilpotent elements are still central here.) In this case, $(x + r)^n = x^n + r^n.$ So $\alpha_n(x) = x^n$ is periodic for every integer $n \geq 2,$ and the periods of $\alpha_n$ are the nilpotent elements with index at most $n.$ It follows that a weakly periodic ring in which $R\cdot\textup{Nil}(R) = 0 = \textup{Nil}(R)\cdot R$ is a commutative periodic ring, albeit not necessarily torsion.

We should emphasize that $R$ being a two-sided annihilator of $\textup{Nil}(R)$ is pivotal to the argument that $R$ is nilperiod. For instance, consider the Klein four-group $\mathbb{V} = \{0, a, b, c\}$ presented additively so that $R^+ = \Z_2 \oplus \Z_2$ and endowed with the multiplication given by $0x = cx = 0$ and $ax = bx = x$ for all $x \in \mathbb{V}.$ Observe that $\text{Pot}(R) = \{a, b\}$ and $\text{Nil}(R) = \{0, c\},$ and yet the power maps over $R$ are merely ``quasiperiodic'' over $\text{Pot}(R)$ in the sense that $(x + c)^n = x^n + c$ for all $x \in R\backslash\{0, c\}$ and $n \in \N.$ This example due to Bell \cite{Bell} is notable for the one-sided orthogonality of $\textup{Nil}(R)$ and $\textup{Pot}(R).$ Various aspects of such rings are discussed in \cite{DY} and \cite{HTY}.

\subsection{Corbas ($p, k, \phi$)-rings}

\renewcommand{\arraystretch}{1.5}
Let us return to Example 2.2, where
\begin{equation}
(a, b)^n = \left\{\begin{array}{ll} (a^n, na^{n - 1}b) & \text{if}\ \phi\ \text{is the identity automorphism,} \\ \left(a^n, b\phi(a^{n - 1})\dfrac{a^n\phi(a^{-n}) - 1}{a\varphi(a^{-1}) - 1}\right) & \text{otherwise},\end{array}\right.
\end{equation}
for all $a \in \F_{p^k}^{\times},$ $b \in \F_{p^k},$ and $n \in \N.$ Since every non-nilpotent element of $R$ is potent, $\mu_0 = 2.$ To ascertain $\mu_1,$ we need to calculate the smallest positive integer $\mu_1$ such that $a^{1 + \mu_1} = a$ and $(1 + \mu_1)a^{\mu_1} = 1$ for every $a \in \F_{p^k}^{\times}.$ Recall that $\F_{p^k}$ has characteristic $p,$ and the multiplicative group $\F_{p^k}^{\times}$ is a cyclic group of order $p^{k} - 1.$ Hence $\mu_1 = \text{lcm}(p, p^{k} - 1) = p(p^k - 1).$ There are thus $p^{k + 1} - p + 1$ distinct power maps over $R.$

Out of these, $p^k - 1$ are periodic when $\phi$ is the identity. To directly see why, note that the only way for the equation
\[\left((a, b) + (0, y)\right)^n = (a^n, na^{n - 1}(b + y)) = (a, b)^n = (a^n, na^{n - 1}b)\]
to hold over all of $R$ is for $n$ to be a multiple of $p,$ in which case $(a, b)^n = (a^n, 0).$ This amount matches the summation derived in Theorem \ref{periodicpowermapbound} since every nonzero nilpotent element of $R$ has order $p$ and so
\begin{equation}
\sum_{\substack{d\, |\, p \\ d \neq 1}} (-1)^{\omega(d) - 1} \left\lfloor\frac{\mu_0 + \mu_1 - 1}{d}\right\rfloor = (-1)^{\omega(p)-1}\left\lfloor\frac{p^{k+1} - p + 1}{p}\right\rfloor = p^k - 1.
\end{equation}

\subsection{Galois Rings}

Let $R = \text{GR}(p^k, d),$ the unique Galois extension of $\Z/p^k\Z$ of degree $d,$ which is a local ring of characteristic $p^k.$ It is well-known that the unique maximal ideal of $R$ is the principal ideal $(p),$ which is entirely comprised of all multiples of $p,$ and that every non-nilpotent element is a unit. Furthermore, $R^{\times} \cong G_1 \times G_2,$ where $G_1$ is a cyclic group of order $p^d - 1$ and
\begin{equation}
\label{groupofunits}
G_2 \cong \left\{\begin{array}{ll} C_2 \times C_{2^{k - 2}} \times (C_{2^{k - 1}})^{d - 1} & \text{if}\ p = 2\ \text{and}\ k \geq 3, \\ (C_{p^{k - 1}})^d & \text{otherwise}.\end{array}\right.
\end{equation}
If $d = 1,$ then $R \cong \Z/p^k\Z.$ This case is discussed in the next subsection. If $k = 1,$ then $R \cong \F_{p^d},$ which by Corollary \ref{jring} has $p^d - 1$ distinct power maps, none of which are periodic. Finally, if $d, k > 1,$ then $\mu_0 = k$ and $\mu_1 = \text{lcm}(p^d - 1, p^{k - 1}) = p^{k - 1}(p^d - 1).$ There are thus $p^{k - 1}(p^d - 1) + k - 1$ distinct power maps over $R$ in this case, and since $R$ is a finite commutative ring with unity, we can apply Theorem \ref{periodicpowermapbound} to see that
\begin{equation}
\mu_P(\text{GR}(p^k, d)) = (-1)^{\omega(p) - 1}\left\lfloor \frac{p^{k - 1}(p^d - 1) + k - 1}{p}\right\rfloor = p^{k - 2}(p^d - 1) + \left\lfloor\frac{k - 1}{p}\right\rfloor
\end{equation}

\subsection{The integers modulo $n$}

The enumeration of the distinct power maps over $\Z/n\Z$ is a folklore result, but we include it below for completeness. On the other hand, a description of the periodic power maps, while elementary, is perhaps new or at least little-known.

Let $n$ be a positive integer and let $p_1^{\beta_1}p_2^{\beta_2} \cdots p_v^{\beta_v}$ be its prime factorization. By the Chinese Remainder Theorem,
\begin{equation}
\Z/n\Z \cong \Z/p_1^{\beta_1}\Z\times\Z/p_2^{\beta_2}\Z\times\cdots\times\Z/p_v^{\beta_v}\Z.
\end{equation}
Based on (\ref{groupofunits}) and the discussion surrounding it, $\mu_1(\Z/p_i^{\beta_i}\Z) = \lambda(p_i^{\beta_i}),$ where $\lambda$ is the Carmichael function defined on prime powers $p^{\beta}$ by
\[\lambda(p^{\beta}) = \left\{\begin{array}{ll} 2^{\beta - 2} & \text{if}\ p = 2\ \text{and}\ \beta \geq 3, \\ p^{\beta - 1}(p - 1) & \text{otherwise}. \end{array}\right.\]
It follows that the exponential period of $\Z/n\Z$ is
\begin{equation}
\mu_1(\Z/n\Z) = \text{lcm}\!\left(\lambda(p_1^{\beta_1}), \lambda(p_2^{\beta_2}), \ldots, \lambda(p_v^{\beta_v})\right) = \lambda(n),
\end{equation}
whereas the onset of exponential periodicity is $E(n) := \max\beta_i,$ that is, the index of
\begin{equation}
\text{Nil}(\Z/n\Z) \cong \text{Nil}\!\left(\Z/p_1^{\beta_1}\Z\right)\times\text{Nil}\!\left(\Z/p_2^{\beta_2}\Z\right)\times\cdots\times\text{Nil}\!\left(\Z/p_v^{\beta_v}\Z\right).
\end{equation}
There are thus $\lambda(n) + E(n) - 1$ distinct power maps over $\Z/n\Z.$ Table 2 collects these statistics for some simple values of $n.$ This settles Blomberg and Whitmore's conjectures for sequence A109746 on the OEIS \cite{OEIS}.

\begin{table}
\centering
\caption{Some power map enumerations in $\Z/n\Z$}
\resizebox{\columnwidth}{!}{%
\begin{tabular}{l | l  l  l  l}
\hline
\textit{Value of $n$} & $\lambda(n)$ & $E(n)$ & \# \textit{of power maps} & \# \textit{that are periodic} \\ \hline \hline
prime $p$ & $p - 1$ & 1 & $p - 1$ & 0 \\
$2p$ (for odd prime $p$) & $p - 1$ & 1 & $p - 1$ & 0 \\
$\prod\limits_{\text{prime factors}\ p_i} p_i$ & lcm$\{p_i - 1\}$ & 1 & lcm$\{p_i - 1\}$ & 0 \\
$p^2$ (for any prime $p$) & $p^2 - p$ & 2 & $p^2 - p + 1$ & $p - 1$ \\
$2^k, k > 2$ & $2^{k - 2}$ & $k$ & $2^{k - 2} + k - 1$ & $2^{k - 3} + \left\lfloor \frac{k - 1}{2} \right\rfloor$ \\
$p^k, k > 2$ (for odd prime $p$) & $p^{k - 1}(p - 1)$ & $k$ & $p^{k - 1}(p - 1) + k - 1$ & $p^{k - 2}(p - 1) + \left\lfloor \frac{k - 1}{p} \right\rfloor$ \\ \hline
\end{tabular}
}
\end{table}

To help us see which power maps over $\Z/n\Z$ are periodic, we recall that $\Z/p_i^{\beta_i}\Z$ is a field precisely when $\beta_i = 1.$ For this reason, $p_i$ divides the order of some nonzero nilpotent element of $\Z/n\Z$ if and only if $\beta_i > 1$ (which is equivalent to the condition that $p_i\, |\, (n/\text{rad}(n))$). Hence, by Theorem \ref{periodicpowermapbound},
\begin{equation}
\mu_P(\Z/n\Z) = \sum_{\substack{d\, |\, \textup{rad}(n/\text{rad}(n)) \\ d \neq 1}} (-1)^{\omega(d) - 1} \left\lfloor\frac{\lambda(n) + E(n) - 1}{d}\right\rfloor.
\end{equation}

\subsection{Matrix rings over fields}

Consider $R = \mathcal{M}_n(\F_q),$ where $q$ is a prime power $p^k.$ Almkvist \cite{Almkvist} proved that $\mu_1(\mathcal{M}_n(\F_2))$ is the exponent of $\text{GL}(n, \F_2).$ Here we modify Almkvist's methods to show that $\mu_1(\mathcal{M}_n(\F_q))$ is the exponent of $\text{GL}(n, \F_q)$ and $\mu_0(\mathcal{M}_n(\F_q)) = n$ in general.

Let $\tau \in R,$ let $\psi(t) \in \F_q[t]$ be the minimal polynomial of $\tau,$ and let $ \psi_1(t)^{\beta_1}\cdots\psi_v(t)^{\beta_v}$ be the prime factorization of $\psi(t).$ Treat $\tau$ as a linear operator of the vector space $V := \F_q^n$ and set $V_i = \ker \psi_i(\tau)^{\beta_i}$ for each $i.$ Then $V = \bigoplus_{i = 1}^{v} V_i$ and each $V_i$ is invariant under $\tau.$ For each $i,$ let $\tau_i: V_i \rightarrow V_i$ be the restriction of $\tau$ to $V_i$ so that $\psi_i(\tau_i)^{\beta_i}$ vanishes. It suffices to flesh out the dynamics of the sequence $\{\tau_i^j\}_{j = 1}^{\infty}.$

If $\psi_i(t) = t,$ then $\tau_i$ is nilpotent and there is nothing more to prove, so assume instead that $\psi_i(t) \neq t.$ Suppose that $\deg\psi_i(t) = m.$ Then $\psi_i(t)\, |\,(t^{q^m} - t)$ since the product of all monic irreducible polynomials in $\F_q[t]$ with degree dividing $m$ is $t^{q^m} - t.$ Hence $(\tau_i^{q^m} - \tau_i)^{\beta_i} = 0,$ from which it follows that $\tau_i$ is potent. Moreover, because $\beta_i \leq n \leq p^{\lceil \log_p n\rceil},$ we have that
\begin{equation}
(t^{q^m} - t)^{\beta_i}\, |\, (t^{q^m} - t)^{p^{\lceil \log_p n \rceil}} = t^{p^{\lceil \log_p n \rceil}q^m} - t^{p^{\lceil \log_p n \rceil}}.
\end{equation}
It follows that
\begin{equation}
\tau_i^{p^{\lceil \log_p n \rceil}q^m} = \tau_i^{p^{\lceil \log_p n \rceil}},
\end{equation}
and so the exponential periodicity of $R$ is a divisor of
\begin{equation}
\label{glnexponent}
\text{lcm}\{p^{\lceil \log_p n \rceil}(q - 1), p^{\lceil \log_p n \rceil}(q^2 - 1), \ldots, p^{\lceil \log_p n \rceil}(q^n - 1)\} = p^{\lceil \log_p n \rceil}\text{lcm}\{q - 1, q^2 - 1, \ldots, q^n - 1\}.
\end{equation}
Yet (\ref{glnexponent}) is precisely the exponent of $\text{GL}(n, \F_q)$ (cf \cite[Corollary 1]{Darafsheh}), and so this yields $\mu_1(R).$ Finally, $\mu_0(R)$ is the largest possible index of a nilpotent operator on a subspace of $V,$ which is $n.$

We finish with a sketch of a constructive proof for the nonexistence of periodic power maps over $\mathcal{M}_n(K)$ that applies to all fields. Let $A \neq 0$ be nilpotent. The Jordan canonical form $J$ of $A$ is a strictly lower triangular matrix with its unital entries lying on the subdiagonal. Now let $C$ be the companion matrix of the polynomial $p(t) = t^n - 1.$ The only nonzero entries of $C$ are at entry $(1, n),$ which is 1, and the subdiagonal completely populated by 1, and so $C - J$ is singular due to having zero-rows. Let $S$ be the change of basis matrix for which $A = SJS^{-1}.$ Then for every $m \in \N,$
\begin{equation}
(SCS^{-1} - A)^m = S(C - J)^mS^{-1}
\end{equation}
is singular whereas $(SCS^{-1})^m = SC^mS^{-1}$ is not.

\section{Concluding Remarks}

While we have explicitly specified the periodic power maps over finite commutative rings with unity, things quickly turn more exotic in broader classes of rings. For instance, Example \ref{hypothesisrejection} shows that $\mu_P$ can potentially reach the upper bound of $\mu_1 + \mu_0 - 2$ where every non-identity power map is periodic. In any event, we pose the following question.

\begin{question}
What is the distribution of $\mu_P$ over random finite commutative rings?
\end{question}

Investigations into the oscillatory behavior of generic power maps over rings can proceed in at least two directions. One is taxonomic: can we achieve at least a partial characterization of the nilperiod rings? Another is combinatorial: is it possible to detect and enumerate the periodic power maps over any periodic ring? To the author's surprise, these questions appear mostly unexplored. We hope that this article generates some interest in what may potentially be a rich and attractive topic.


\bibliographystyle{plain}

\begin{thebibliography}{99}

\bibitem{AHY}
{Abu-Khuzam, H. and Hasanali, M. and Yaqub, A.},
{Weakly periodic rings with conditions on commutators},
{\sl Acta Math. Hungar.},
{\bf 71}, (1996), nos. 1-2, {145--153}.

\bibitem{Almkvist}
{Almkvist, Gert},
{Powers of a matrix with coefficients in a Boolean ring},
{\sl Proc. Amer. Math. Soc.},
{\bf 53}, (1975), no. 1, {27--31}.

\bibitem{Amitsur}
{Amitsur, S. A.},
{A generalization of {H}ilbert's {N}ullstellensatz},
{\sl Proc. Amer. Math. Soc.},
{\bf 8}, (1957), {649--656}.

\bibitem{Bell}
{Bell, Howard E.},
{A commutativity study for periodic rings},
{\sl Pacific J. Math.},
{\bf 70}, (1977), no. 1, {29--36}.

\bibitem{BK}
{Bell, Howard E. and Klein, Abraham A.},
{On finiteness, commutativity, and periodicity in rings},
{\sl Math. J. Okayama Univ.},
{\bf 35}, (1993), {181--188}.

\bibitem{BT}
{Bell, Howard E. and Tominaga, Hisao},
{On periodic rings and related rings},
{\sl Math. J. Okayama Univ.},
{\bf 28}, (1986), {101--103}.

\bibitem{BHL}
{Birkenmeier, Gary F. and Heatherly, Henry E. and Lee, Enoch K.},
{Completely prime ideals and associated radicals},
{\sl Ring theory ({G}ranville, {OH}, 1992)},
{102--129}, {World Sci. Publ., River Edge, NJ}, (1993).

\bibitem{Braun}
{Braun, Amiram},
{The nilpotency of the radical in a finitely generated {PI}
              ring},
{\sl J. Algebra},
{\bf 89}, (1984), no. 2, {375--396}.

\bibitem{Chacron}
{Chacron, M.},
{On a theorem of Herstein},
{\sl Canadian J. Math.},
{\bf 21}, (1969), {1348--1353}.

\bibitem{Corbas}
{Corbas, Basil},
{Rings with few zero divisors},
{\sl Math. Ann.},
{\bf 181}, (1969), {1--7}.

\bibitem{Darafsheh}
{Darafsheh, M. R.},
{Order of elements in the groups related to the general linear group},
{\sl Finite Fields Appl.},
{\bf 11}, (2005), no. 4, {738--747}.

\bibitem{DY}
{Du, Xiankun and Yi, Qi},
{On periodic rings},
{\sl Int. J. Math. Math. Sci.},
{\bf 25}, (2001), no. 6, {417--420}.

\bibitem{GTY}
{Grosen, Julie and Tominaga, Hisao and Yaqub, Adil},
{On weakly periodic rings, periodic rings and commutativity
              theorems},
{\sl Math. J. Okayama Univ.},
{\bf 32}, (1990), {77--81}.

\bibitem{Herstein1}
{Herstein, I. N.},
{A note on rings with central nilpotent elements},
{\sl Proc. Amer. Math. Soc.},
{\bf 5}, (1954), {620}.

\bibitem{Herstein2}
{Herstein, I. N.},
{Power maps in rings},
{\sl Michigan Math. J.},
{\bf 8}, (1961), {29--32}.

\bibitem{Herstein3}
{Herstein, I. N.},
{A remark on rings and algebras},
{\sl Michigan Math. J.},
{\bf 10}, (1963), {269--272}.

\bibitem{HJL}
{Hwang, Seo Un and Jeon, Young Cheol and Lee, Yang},
{Structure and topological conditions of {NI} rings},
{\sl J. Algebra},
{\bf 302}, (2006), no. 1, {186--199}.

\bibitem{HTY}
{Hirano, Yasuyuki and Tominaga, Hisao and Yaqub, Adil},
{On rings in which every element is uniquely expressible as a
              sum of a nilpotent element and a certain potent element},
{\sl Math. J. Okayama Univ.},
{\bf 30}, (1988), {33--40}.

\bibitem{Jacobson}
{Jacobson, N.},
{Structure theory for algebraic algebras of bounded degree},
{\sl Ann. of Math. (2)},
{\bf 46}, (1945), {695--707}.

\bibitem{Kaplansky}
{Kaplansky, Irving},
{Rings with a polynomial identity},
{\sl Bull. Amer. Math. Soc.},
{\bf 54}, (1948), {575--580}.

\bibitem{KPV}
{Kedlaya, Kiran S. and Poonen, Bjorn and Vakil, Ravi},
{The {W}illiam {L}owell {P}utnam {M}athematical {C}ompetition, 1985--2000},
Mathematical Association of America, Washington D.C. (2002),
ISBN 0-88385-807-X.

\bibitem{Marks}
{Marks, Greg},
{A taxonomy of 2-primal rings},
{\sl J. Algebra},
{\bf 266}, (2003), no.2, {494--520}.

\bibitem{OEIS}
{\sl The On-Line Encyclopedia of Integer Sequences},
{available electronically at \texttt{oeis.org}}.

\bibitem{Yaqub}
{Yaqub, Adil},
{Structure of weakly periodic rings with potent extended
              commutators},
{\sl Int. J. Math. Math. Sci.},
{\bf 25}, (2001), no. 5, {299--304}.

\end{thebibliography}

\end{document}